\documentclass[10pt]{article}

\makeatletter  
\newcommand{\xleftrightarrows}[2][]{\mathrel{%
 \raise.40ex\hbox{$  
       \ext@arrow 3095\leftarrowfill@{\phantom{#1}}{#2}$}%
 \setbox0=\hbox{$\ext@arrow 0359\rightarrowfill@{#1}{\phantom{#2}}$}%
 \kern-\wd0 \lower.40ex\box0}}  
 
\def\leftrightarrowfill@{
 \arrowfill@\leftarrow\relbFwar\rightarrow%
}  
\makeatother

     \newtheorem{theorem}{Theorem}[section]
\newtheorem{lemma}[theorem]{Lemma}

\newenvironment{proof}[1][Proof]{\begin{trivlist}
\item[\hskip \labelsep {\bfseries #1}]}{\end{trivlist}}

\newenvironment{remark}[1][Remark]{\begin{trivlist}
\item[\hskip \labelsep {\bfseries #1}]}{\end{trivlist}}

\newcommand{\qed}{\nobreak \ifvmode \relax \else
      \ifdim\lastskip<1.5em \hskip-\lastskip
      \hskip1.5em plus0em minus0.5em \fi \nobreak
      \vrule height0.75em width0.5em depth0.25em\fi}

\usepackage{minitoc}
\usepackage{epsfig}
\usepackage{amssymb,amsmath}  
\usepackage{graphicx}
\usepackage{times}
\usepackage{dcolumn}
\usepackage{bm}
\usepackage{longtable}

\def\beq{\begin{eqnarray}}
\def\eeq{\end{eqnarray}}

\begin{document}

\mbox{}
\thispagestyle{empty}
 \renewcommand{\thesection}{\arabic{section}}
\renewcommand{\thesubsection}{\arabic{subsection}}
\begin{flushleft}
{\Large
\textbf{Dynamics and bifurcations in a simple quasispecies model of tumorigenesis}
}
\vspace{0.3cm}
\\
Vanessa Castillo$^{1}$,
J. Tom\'as L\'azaro$^{1,\ast}$
Josep Sardany\'es$^{2,3,\ast}$, 
\vspace{0.1cm}
\\
\small{
1. Universitat Polit\`ecnica de Catalunya, Barcelona, Spain
\\
2. ICREA-Complex Systems Lab, Department of Experimental and Health Sciences, Universitat Pompeu Fabra, Dr. Aiguader 88, 08003 Barcelona, Spain
\\
3. Institut de Biologia Evolutiva (CSIC-Universitat Pompeu Fabra), Passeig Maritim de la Barceloneta 37, 08003 Barcelona, Spain
\\
$\ast$ Corresponding authors. E-mail: jose.tomas.lazaro@upc.edu; josep.sardanes@upf.edu.}
\end{flushleft}

\begin{center}
\section*{Abstract}
\end{center}
Cancer is a complex disease and thus is complicated to model. However, simple models that describe the main processes involved in tumoral dynamics, e.g., competition and mutation, can give us clues about cancer behaviour, at least qualitatively, also allowing us to make predictions. Here we analyze a simplified quasispecies mathematical model given by differential equations describing the time behaviour of tumor cells populations with different levels of genomic instability. We find the equilibrium points, also characterizing their stability and bifurcations focusing on replication and mutation rates. We identify a transcritical bifurcation at increasing mutation rates of the tumor cells population. Such a bifurcation involves an scenario with dominance of healthy cells and impairment of tumor populations. Finally, we characterize the transient times for this scenario, showing that a slight increase beyond the critical mutation rate may be enough to have a fast response towards the desired state (i.e., low tumor populations) during directed mutagenic therapies.
\vspace*{1cm}

\section{Introduction}
Cancer progression is commonly  viewed  as a cellular microevolutionary process \cite{CAIRNS1975,MERLO2006}.  Genomic instability,  which seems to be a common trait  in most types  of cancer  \cite{CAHILL1999},  is a key factor responsible for tumor progression since allows a Darwinian exploratory process required to overcome selection barriers. By displaying either high levels of mutation or chromosomal aberrations, cancer cells can generate a progeny of highly diverse phenotypes able to evade such barriers \cite{LOEB2001}. Genomic instability refers to an increased tendency of alterations in the genome during the life cycle of cells. Normal cells display a very low rate of mutation ($1.4 \times 10^{-10}$ changes per nucleotide and replication cycle). Hence, it has been proposed that the spontaneous mutation rate in normal cells is not sufficient to account for the large number of mutations found in human cancers. Indeed, studies of mutation frequencies in microbial populations, in both experimentally induced stress and clinical cases, reveal that mutations that inactivate mismatch repair genes result in $10^2 - 10^3$ times the background mutation rate \cite{OLIVER2000,BJEDOV2003,SNIE1997}. Also, unstable tumors exhibiting the so-called mutator phenotype \cite{LOEB2001} have rates that are at least two orders of magnitude higher than in normal cells \cite{ANDERSON2001, BIELAS2006}. This difference leads to cumulative mutations and increased levels of genetic change associated to  further failures in genome  maintenance  mechanisms \cite{Hoeijmakers}.  The  amount  of instability  is, however,  limited by  too high levels of instability, which have been suggested to exist in tumor progression \cite{CAHILL1999}, thus indicating that thresholds for  instability must exist.  In fact, many anti-cancer therapies  take indirectly advantage  of increased genomic instability, as is the case  of mitotic spindle alteration by taxol or DNA damage by radiation or by alkilating agents.

The mutator phenotype is the result of mutations in genes which are responsible of preserving genomic stability, e.g., BRCA1 (breast cancer 1), BLM (bloom syndrome protein), ATM (ataxia telangiectasia mutated) or gene protein P53 (which is involved in multitude of cellular pathways involved in cell cycle control, DNA repair, among others). This mutator phenotype undergoes increases in mutation rates and can accelerate genetic evolution in cancer cells that can ultimately drive to tumor progression \cite{LOEB2001}. As mentioned, genomic instability is a major driving force in tumorigenesis. Tumorigenesis can be viewed as a process of cellular microevolution in which individual preneoplastic or tumor cells acquire mutations that can increase proliferative capacity and thus confer a selective advantage in terms of growth speed. The rate of replication of tumor cells can increase due to mutations in both tumor suppressor genes, e.g., APC (adenomatous polyposis coli) or P53; and oncogenes, e.g., RAS (rat sarcoma) or SRC. Tumor suppressor genes protect cells from one step on the path to cancer and oncogenes are genes that, when mutated, have the potential to cause cancer. In terms of population dynamics, alterations in both types of genes drive to neoplastic process through increases in cancer cells numbers. Mutations in replication-related genes that confer an increase of fitness and thus a selective advantage are named driver mutations \cite{BENEDIKT2014}. This evolutionary process allows tumor cells to escape the restrictions that limit growth of normal cells, such as the constraints imposed by the immune system, adverse metabolic conditions or cell cycle checkpoints. 

The iterative process of mutation and selection underlying tumor growth and evolution promotes the generation of a diverse pool of tumor cells carrying different mutations and chromosomal abnormalities. In this sense, it has been suggested that the high mutational capacity of tumor cells, together with an increase of proliferation rates, may generate a highly diverse population of tumor cells similar to a quasispecies \cite{SOLE2001,BRUMER2006}. A quasispecies is a ``cloud" of genetically related genomes around the so-called master sequence, which is at the mutation-selection equilibrium \cite{Eigen1971,Eigen1979}. Due to the heterogeneous population structure, selection does not act on a single mutant but on the quasispecies as a whole. The most prominent examples of a quasispecies are given by RNA viruses (e.g. Hepatitis C virus \cite{SOLE2006}, vesicular stomatitis virus \cite{Marcus1998}, the human immunodeficiency virus type 1 \cite{CICHUTEK1992}, among others). 

An important concept in quasispecies theory is the so-called error threshold \cite{Eigen1971,Eigen1979}. The error threshold is a phenomenon that involves the loss of information at high mutation rates. According to Eigen's original formulation, a quasispecies can remain at equilibrium despite high mutation rates, but the surpass of the critical mutation rate will upset this balance since the master sequence itself disappears and its genetic information is lost due to the accumulation of errors. It has been suggested that many RNA viruses replicate near the error threshold \cite{DOMINGO2001}. Another important concept in quasispecies theory is lethal mutagenesis. As a difference from the error threshold (which is a shift in sequence space), lethal mutagenesis is a process of demographic extinction due to an unbearable number of mutations \cite{BULL2007,WYLIE2012}. Most basically, it requires that deleterious mutations are happening often enough that the population cannot maintain itself, but it is otherwise no different from any other extinction process in which fitness is not great enough for one generation of individuals to fully replace themselves in the next generation. In simple words, increased mutagenesis could impair the maintainance of a quasispecies due to the crossing of the error threshold or due to lethal mutagenesis. 

Quasispecies theory has provided a population-based framework for understanding RNA viral evolution \cite{LAURING2010}. These viruses replicate at extremely high mutation rates and exhibit significant genetic diversity. This diversity allows a viral population to rapidly adapt to dynamic environments and evolve resistance to vaccines and antiviral drugs. As we previously mentioned, several features have been suggested to be shared between RNA viruses and tumors \cite{SARD}, at least qualitatively. One is the presence of high levels of heterogeneity, both at genotype and phenotype levels. Typically, cancer cells suffer mutations affecting cell communication, growth and apoptosis (i.e., programmed cell death). Accordingly, escape from the immune system (and other selection barriers) operates in both RNA viruses and tumors. Viruses use antigenic diversity whereas tumors evade the immune system by loosing their antigens through mutation, or making use of antigenic modulation and/or tumor-induced immune suppression. Even more, similarly to RNA viruses, genetic instability in cancer cells will have detrimental effects on cells fitness, since most random mutations are likely to be harmful. As indicated by Cahill et al. \cite{CAHILL1999}, the best chance of cure advanced cancers might be a result of tumor genetic instability: cancer cells are more sensitive to stress-inducing agents. In this sense, possible therapies increasing mutation of tumor cells could push this populations towards the error threshold or induce lethal mutagenesis. This is the topic that we will address in this work by using a mathematical model describing the dynamics of competition between different cell populations with different levels of genomic instability. Specifically, we will find equilibrium points, characterizing their stability, focusing on possible changes in the stability as a function of the mutation rate and the fitness of cells populations. We will also investigate the transient times of the system to reach a given equilibrium state. 

Each of the mathematical results will be given in terms of the mutation rates or the replication fidelity of the cells. This is important to understand how tumor populations behave. As we previously mentioned, unstable cells may reach an error threshold and start loosing genetic information until its population decreases or, even more, disappears. We will see how mutation rates affect this error threshold. Furthermore, we will see how they affect to the velocity of the tumor population to reach an equilibrium point, which is interesting since mutation rates can be modified through drugs (similarly to RNA viruses \cite{CROTTY2001}) or radiotherapy.

\section{The model}
\label{model}
In this section we introduce the model by Sol\'e and Deisboeck \cite{SOLE2004} , which describes the competitive behaviour between cell populations with different levels of genomic instability. The model is given by the following set of differential equations:

\begin{figure}[!ht]
\centering
\includegraphics[scale=0.5]{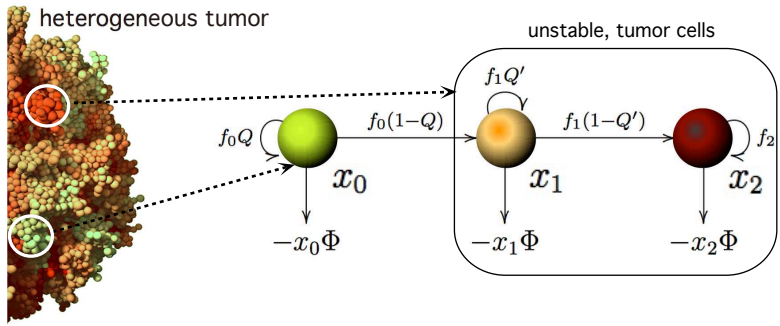}
\caption{ Schematic diagram of the investigated system. We consider a simple model of tumor growth describing the dynamics of competition between different cell populations with different levels of genomic instability. Three different cell populations are considered: $x_0$ represents anomalous growth but low genomic instability, and both $x_1$ and $x_2$ have larger genomic instabilities (see section \ref{model} for a description of the model parameters).}
\label{diagram}
\end{figure}

\begin{eqnarray}
\left\lbrace\begin{array}{lll}
\dot{x_0} &=& f_0Qx_0 - \Phi(x_0,x_1,x_2)x_0, \nonumber\\[5pt]
\dot{x_1}& =& f_0(1-Q)x_0 + f_1Q'x_1 -\Phi(x_0,x_1,x_2)x_1, \\[5pt]
\dot{x}_2^i &=  &f_1(1-Q')q'_ix_1 + \sum_{j=1}^{n}f_2^j\mu_{ij} x_2^j - 	\Phi(x_0,x_1,x_2)x_2^i.\nonumber 
\end{array}\right.\label{eq1}
\end{eqnarray}

\bigskip

\noindent Variable $x_0$ is the fraction of cells with anomalous growth but no genetic instability; $x_1$ is the fraction of cells derived from $x_0$ by mutation that allows genetic instability; and $x_2^i$ for $i=1,2...,n$, is the fraction of tumor mutant cells that can be generated from $x_1$ due to mutation (see Fig.~\ref{diagram}). Thus, $x_0$, $x_1$ and $x_2^i$ denote the fraction of each population and, therefore, the sum of all these variables must be exactly one, i.e., $x_0 + x_1+ \sum_{i=1}^n x_2^i = 1$. Moreover, the probability of mutation from $x_1$ to $x_2^i$ can be denoted by $\mu_{x_{1} \rightarrow x_2^i}$ and is given by $(1-Q')q'_i$. Notice that $\sum_{i=1}^n q'_i = 1$. Coefficients $\mu_{ij}$ denote the mutation rate (or probability) from $x_2^j$ to $x_2^i$. The probability of error-free replication of $x_2^i$ is represented by $\mu_{ii}$. Additionally, cross-mutations connect the different subclones of $x_2$ through the term $\sum_{j = 1}^n f_2^j \mu_{ij}x_2^j$, being $f_j$ the growth rate of each population and $f_2^i$ the growth rate of each subpopulation of $x_2$. These $f_j$ and $f_2^i$ are known as the fitness of each population. The greater is the fitness of a population, the greater is its replicative rate. Finally, the term $\Phi(\textbf{x})$ is the average fitness of the population vector $\textbf{x} = (x_0, x_1, x_2^1, x_2^2,...,x_2^n)$, i.e., $\Phi(\textbf{x}) = f_0x_0+f_1x_1+\sum_{i=1}^n f_2^ix_2^i$.  $\Phi(\textbf{x})$ is also known as the ``constant population constraint" and it ensures that the population remains constant, also introducing competition between the three populations of cells. Although the model does not explicitly consider environmental constraints, such as blood supply, hypoxia or acidosis, they can be considered as implicitly introduced through the set $\{f_j\}_{j=0,1,2}$ \cite{SOLE2001}. Population $x_0$ mutates to $x_1$ with a probability of $\mu_0=1-Q$. In the same way, $x_1$ mutates in different $x_2^i$-sequences with a probability $\mu_1=1-Q'$. In both cases, $0<Q,Q'<1$, being $Q$ and $Q'$ the copying fidelity during replication for $x_0$ and $x_1$ respectively. So, in this model, mutations from $x_1$ to $x_0$ and from $x_2^i$ to $x_1$ have not been considered. It is also worth noticing that the model with $Q'=1$ is the two-variable quasispecies model \cite{SWETINA1982}.\\ 
The set of equations (\ref{eq1}) can also be written in a matricial way:

\begin{eqnarray}
\begin{pmatrix}
\dot{x_0} \\
\dot{x_1} \\
\dot{x}_2^1\\
\vdots \\
\dot{x}_2^n
\end{pmatrix}
= \left(
\begin{array}{cc}
\begin{array}{cc}
f_0 Q & 0 \\
f_0(1-Q) & f_1Q'\\
0 & f_1(1-Q')q'_1\\
\vdots & \vdots\\
0 &   f_1(1-Q')q'_n
\end{array} &
\begin{array}{c}
\begin{array}{ccc}
0 & \cdots & 0 \\
0 & \cdots & 0 \\
\end{array}\\
\begin{array}{|ccc|}
\hline
&  &\\
& M_{\mu}D_{f_2} & \\
& & \\
\hline
\end{array}
\end{array}
\end{array}\right)
\begin{pmatrix}
x_0 \\
x_1 \\
x_2^1\\
\vdots \\
x_2^n
\end{pmatrix}
-
\Phi(\textbf{x})
\begin{pmatrix}
x_0 \\
x_1 \\
x_2^1\\
\vdots \\
x_2^n
\end{pmatrix},
\end{eqnarray}

\bigskip
where $M_{\mu} = 
\begin{pmatrix}
\mu_{11} & \cdots & \mu_{1n}\\
\vdots &  & \vdots\\
\mu_{n1} & \cdots & \mu_{nn}
\end{pmatrix}
$ and $
D_{f_2} = 
\begin{pmatrix}
f_2^1 & & \\
 & \ddots & \\
 & & f_2^n
\end{pmatrix}.$\\
\\
So, if we set $M$ and $D_f$ as the following matrices
\begin{eqnarray*}
M = \left(
\begin{array}{cc}
\begin{array}{cc}
Q & 0 \\
1-Q & Q'\\
0 & (1-Q')q'_1\\
\vdots & \vdots\\
0 &   (1-Q')q'_n
\end{array} &
\begin{array}{c}
\begin{array}{ccc}
0 & \cdots & 0 \\
0 & \cdots & 0 \\
\end{array}\\
\begin{array}{|ccc|}
\hline
&  &\\
& M_{\mu} & \\
& & \\
\hline
\end{array}
\end{array}
\end{array}\right) 
, \phantom{xx} D_f = 
\begin{pmatrix}
f_0 & & & &\\
 & f_1 & & & \\
 & & f_2^1 & & \\
 & & & \ddots & \\
 & & & & f_2^n
\end{pmatrix},
\end{eqnarray*}\\
then $M$ is a Markov matrix by columns. This means that $\sum_{i=1}^{n+2}M_{ij} = 1$ $\forall j=1\div n+2$. This kind of matrix appears in mathematical models in biology \cite{USHER1971}, economics (e.g., the Markov Switching Multifractal asset pricing model \cite{LUX2011}), telephone networks (the Viterbi algorithm for error corrections \cite{FORNEY1973}) or "rankings" as the PageRank algorithm from Google \cite{SERGEY1998,BRYAN2006}. Therefore, the system can be rewritten as: 
\begin{eqnarray}
\dot{\textbf{x}} = MD_f\textbf{x}-\Phi(\textbf{x})\textbf{x}.
\end{eqnarray}
Our approach to this problem consists on assuming $\{x_2^i\}$ behaves as an average variable $x_2 = \sum_{i=1}^n x_2^i$ \cite{SOLE2004}. As a consequence, only two different mutation rates are involved in such simplified system: $\mu_0 = 1-Q$ and $\mu_1 =1- Q'$. Hence, the set of equations is given by:
\begin{eqnarray}
\left\{
\begin{array}{lll}
\dot{x_0} = f_0Qx_0- \Phi(x_0,x_1,x_2)x_0,

\\ \dot{x_1} = f_0(1-Q)x_0 + f_1Q'x_1 - \Phi(x_0,x_1,x_2)x_1,

\\ \dot{x_2} =  f_1(1-Q')x_1 + f_2x_2 - \Phi(x_0,x_1,x_2)x_2.
\end{array}
\right. \label{eqS}
\end{eqnarray}
In this simplified system, the effect of mutations can be represented by means of a directed graph as shown in Fig.~\ref{diagram}. From now on, let us assume that we always start with a population entirely composed by stable cells, i.e., $x_0(0) = 1$ and $x_1 (0) = x_2 (0) = 0$. Notice that for this particular case every point of the trajectory always verifies $x_0 + x_1 +  x_2 = 1$.

The model by Sol\'e and Deisboeck \cite{SOLE2004} considered $f_2 = \alpha f_1 \phi (\mu_1)$, where $\phi (\mu_1)$ was a decreasing function such that $\phi(0) = 1$, indicating that the speed of fitness decays as mutation increases, and where $\alpha$ was a competition term. The dependence of $f_2$ on $f_1$ and $\mu_1$ was introduced to take into account the deleterious effects of high genetic instability on the fitness of the unstable population $x_2$. In our work we will consider $f_2$ as a constant, being independent of $f_1$ and $\mu_1$. This assumption allows the analysis of a more general scenario where population $x_1$ can produce $x_2$ by mutation, and where $x_2$ can have different fitness properties depending on the selected values of both $f_1$ and $f_2$. For instance, we can analyze the case of deleterious mutations from $x_1$ to $x_2$ with $f_2 < f_1$, the case of neutral mutations from $x_1$ to $x_2$ when $f_2 = f_1$, or the case of increased replication of $x_2$ with $f_2 > f_1$. This latter case would correspond to mutations in driver genes that might confer a selective advantage to the tumor populations $x_2$.

\section{Results and discussion}

\subsection{Equilibrium points}
\label{se:eqpoints}

The system (\ref{eqS}) has three different fixed points. Namely, a first one showing the total dominance of the unstable population, i.e., $x_0$ and $x_1$ go extinct and then $x_2 = 1$. A second one where $x_0$ goes extinct and $x_1$ coexists with $x_2$, and a third possible fixed point, where the three populations coexist. We can not consider a fixed point with $x_0\neq0$, $x_2\neq0$ and $x_1 =0$ because, as shown in Fig.~\ref{diagram}, cells only mutate in one direction and it is not possible to have such scenario. Furthermore, the trivial equilibrium $x_0 = x_1 = x_2 = 0$ is never achieved because $x_0+x_1+x_2=1$.

Let us seek for these three possible fixed points. Then we have to find the solutions of the following system:
\begin{figure}[!ht]
\centering
\includegraphics[scale=0.75]{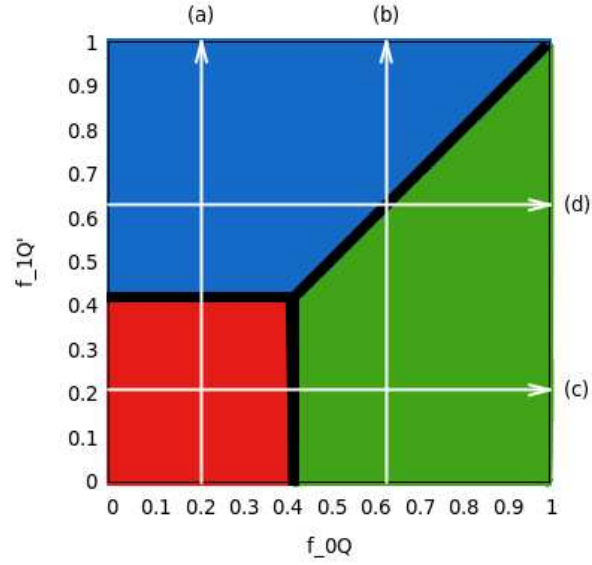}
\caption{Analysis of the parameter space $(f_0 Q, f_1 Q')$ to find the fixed points. Each colored region indicates a different fixed point (smaller red square: $(0,0,1)$; blue upper region: $(0,x_1^*,x_2^*)$; and green region at the right: $(x_0^*,x_1^*,x_2^*)$). The thick black lines indicate that the system does not reach any fixed point. The white arrows are different sections that are studied along this work.}
\label{fixpoints}
\end{figure}

\begin{eqnarray}
\begin{pmatrix}
0 \\
0 \\
0
\end{pmatrix}
=
\begin{pmatrix}
f_0Q & 0 & 0 \\
f_0(1-Q) & f_1Q' & 0\\
0 & f_1(1-Q') & f_2
\end{pmatrix}
\begin{pmatrix}
x_0 \\
x_1 \\
x_2
\end{pmatrix}
-
\Phi(\textbf{x})
\begin{pmatrix}
x_0 \\
x_1 \\
x_2
\end{pmatrix},\label{linearSys}
\end{eqnarray}

\begin{itemize}
\item If we consider $x_0 = x_1 = 0$ and $x_2 \neq 0$, from the system (\ref{linearSys}) we obtain $f_2x_2 - \Phi(x_0,x_1,x_2)x_2 = 0$, which has two possible solutions $x_2 = 0$ or $x_2 = 1$. As we require the sum of the three variables to be 1, then the solution is $(x_0^*,x_1^*,x_2^*)=(0,0,1)$.

\item Let us consider now $x_0 = 0$ and then solve the obtained system of equations:
\begin{eqnarray}
\left\{
	\begin{array}{lll}
	f_1Q' - f_1x_1 - f_2x_2 = 0, \\[5pt]
f_1(1-Q')x_1 + f_2x_2 - f_1x_1x_2 -f_2x_2^2 = 0.
	\end{array}
	\right.
\end{eqnarray}

Its solution is $x_1^* = \frac{f_1Q'-f_2}{f_1-f_2}$ and $x_2^* = \frac{f_1-f_1Q'}{f_1-f_2}$. Notice that, as $x_1$ and $x_2$ are ratios of population, $x_1^*$ and $x_2^*$ must take values from $0$ to $1$. Then, from their expressions and remembering that $0<Q'<1$, we get the conditions: $f_1>f_2$ and $f_1Q'>f_2$.

\item Finally, for the last fixed point we consider $x_0 \neq 0, x_1 \neq 0, x_2 \neq 0$. Then, from the first equation of the system we have $\Phi(x_0,x_1,x_2) = f_0Q$. Hence the new system has this form:
\begin{eqnarray}
 \left\{
\begin{array}{llr}
f_0(1-Q)x_0 + (f_1Q'-f_0Q)x_1 =0,\\[5pt]
f_1(1-Q')x_1 + (f_2-f_0Q)x_2 = 0,\\[5pt]
\end{array}
\right.
\end{eqnarray}
\begin{figure}[!ht]
\begin{minipage}[t]{0.35\textwidth}
\includegraphics[width=\textwidth]{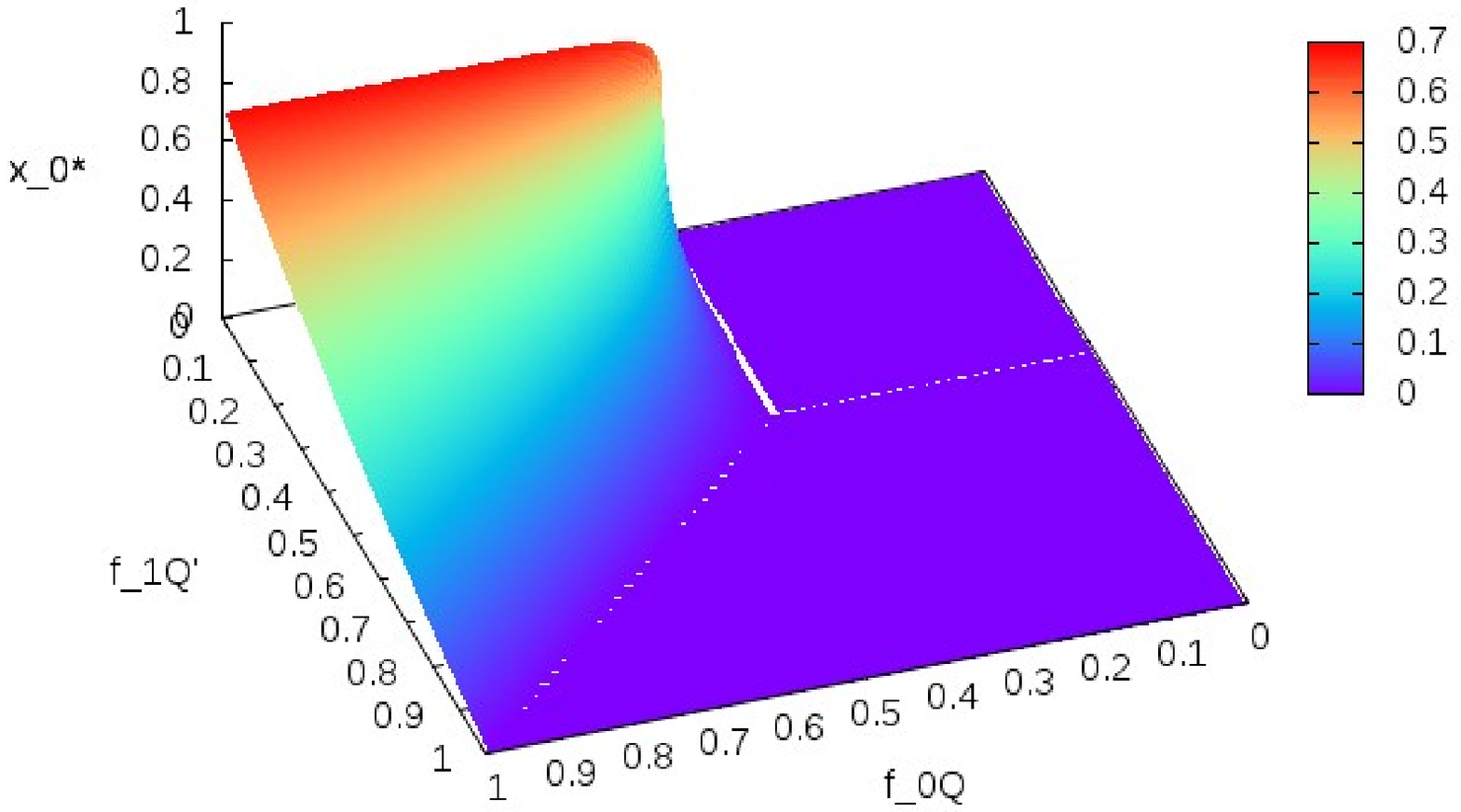}
\end{minipage}
\begin{minipage}[t]{0.35\textwidth}
\includegraphics[width=\textwidth]{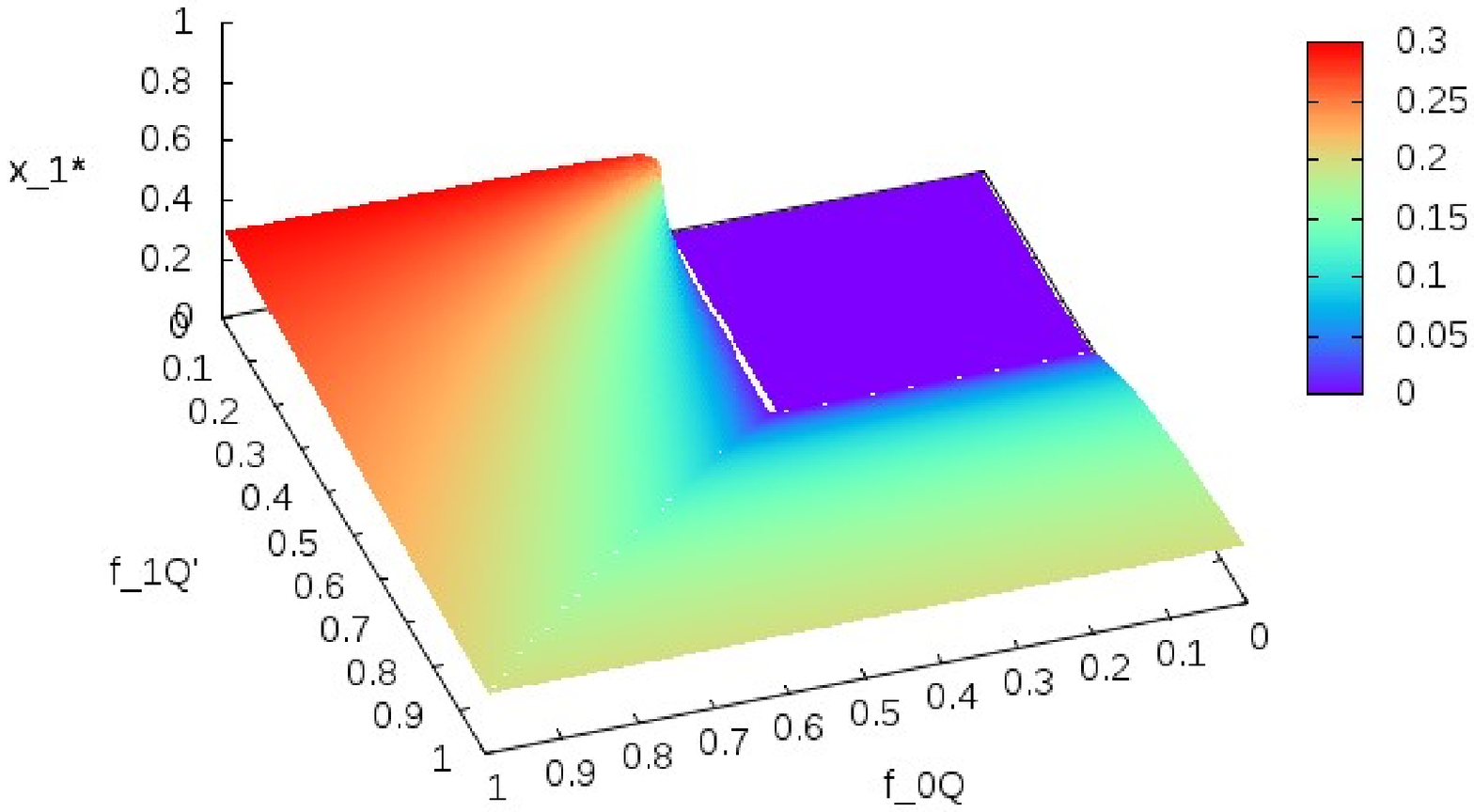}
\end{minipage}
\begin{minipage}[t]{0.35\textwidth}
\includegraphics[width=\textwidth]{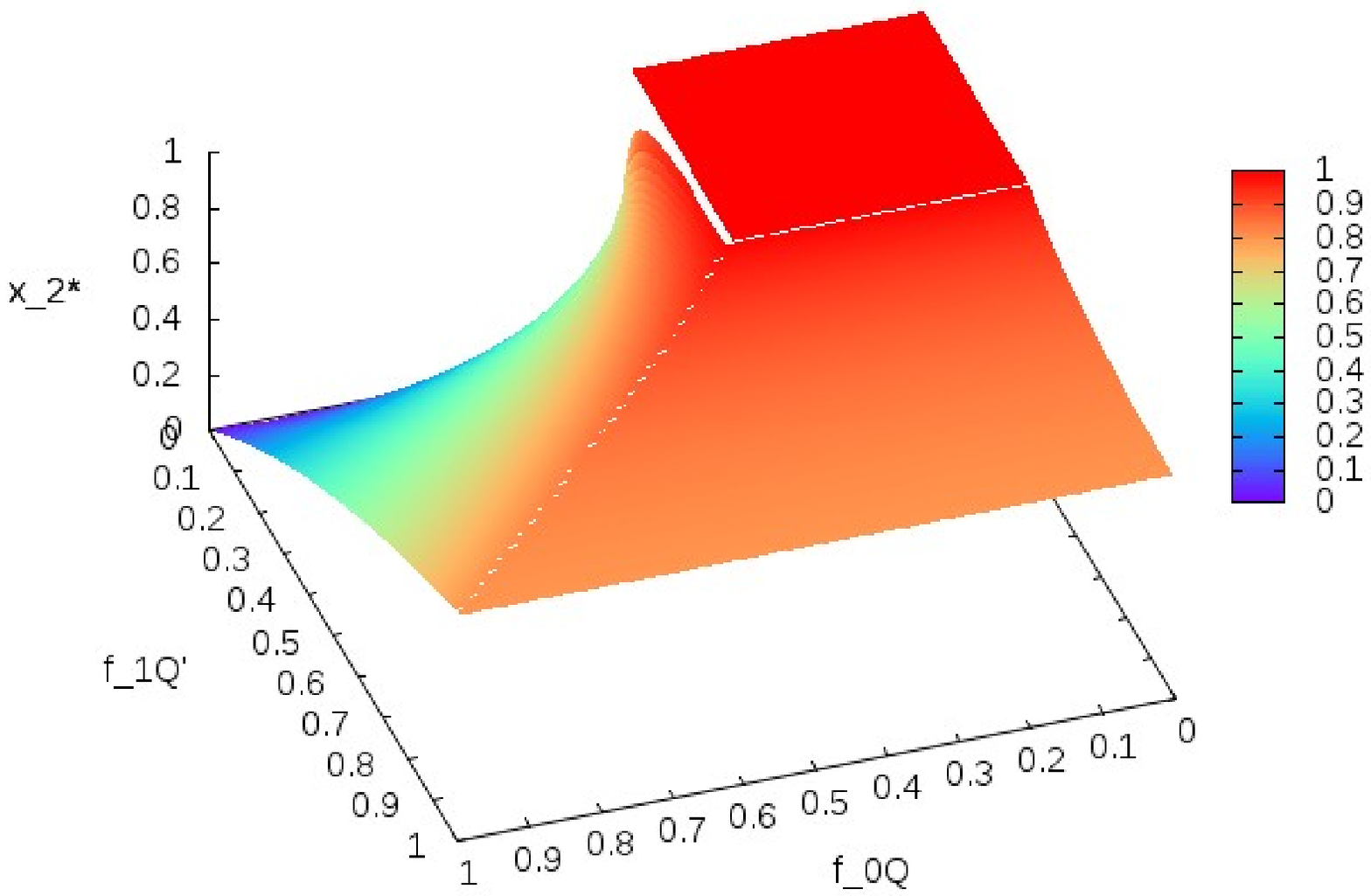}
\end{minipage}
\caption{Equilibrium concentration of each variable in the parameter space ($f_0Q$,$f_1Q'$). We display the equilibrium value for $x_0$ (left panel), $x_1$ (middle panel) and $x_2$ (right panel). Notice that the color bar for each panel is not normalized.}
\label{3d_fix}
\end{figure}

under the constraint $x_0 + x_1 + x_2 = 1$. Then its solution is as follows:
\begin{eqnarray*}
x_0^* = \frac{(f_0Q -f_1Q')(f_0Q-f_2)}{\varphi(f_0,f_1,f_2,Q,Q')},
\end{eqnarray*}
\begin{eqnarray*}
x_1^*= \frac{f_0(1-Q)(f_0Q-f_2)}{\varphi(f_0,f_1,f_2,Q,Q')},
\end{eqnarray*}
and
\begin{eqnarray*}
x_2^*=\frac{f_1(1-Q')f_0(1-Q)}{\varphi(f_0,f_1,f_2,Q,Q')}.
\end{eqnarray*}
where $$\varphi(f_0,f_1,f_2,Q,Q')= (f_0Q -f_1Q')(f_0Q-f_2)+f_0(1-Q)(f_0Q-f_2)+(f_1(1-Q')f_0(1-Q).$$
\end{itemize}

\begin{remark}{{\bf{3.1.}}}
\label{f1f2}
We can consider $f_1 = f_2$ as particular case. If these two fitness parameters are equal, the system (\ref{eqS}) only has two possible fixed points. As we previously mentioned, this case would correspond to the production of neutral mutants from population $x_1$ to $x_2$.
\end{remark}

From now on, we provide numerical results\footnote{The codes used to obtain the results presented in this work are available upon request} of the system (\ref{eqS}). To compute the solutions of this system we have used the \textit{Taylor} software. \textit{Taylor} is an ODE solver generator that reads a system of ODEs and it outputs an ANSI C routine that performs a single step of the numerical integration of the ODEs using the Taylor method. Each step of integration chooses the step and the order in an adaptive way trying to keep the local error below a given threshold, and to minimize the global computational effort \cite{TAYLOR}.

Numerically, we integrate the solution of system (\ref{eqS}) with initial condition $(0,0,1)$. The next Lemma ensures that any point of this orbit satisfies $x_0+x_1+x_2 = 1$. It is used as an accuracy control while integrating the ODE system.
\begin{lemma}
$S = x_0 + x_1 + x_2$ is a first integral of system (\ref{eqS}), that is $\dfrac{dS}{dt}=0$.
\end{lemma}
\begin{proof}
From the equations (\ref{eqS}), we know that
\begin{eqnarray*}
\begin{array}{lcl}
\dfrac{dS}{dt} &=& \dot{x_0}+ \dot{x_1} + \dot{x_2}\\[5pt]
&=&f_0Qx_0+f_0(1-Q)x_0+f_1Q'x_1+f_1(1-Q')x_1+f_2x_2-\Phi(x_0+x_1+x_2)\\[5pt]
&=& \Phi(x_0,x_1,x_2)(1-(x_0+x_1+x_2)).
\end{array}
\end{eqnarray*}
Since the initial condition is $x_0 (0)= 1$ and $x_1(0)=x_2(0)=0$, $1-(x_0+x_1+x_2) = 0$ and the assertion follows. $\phantom{xxxx}$
$\Box$
\end{proof}

\begin{figure}[!ht]
\centering
\includegraphics[scale=0.6]{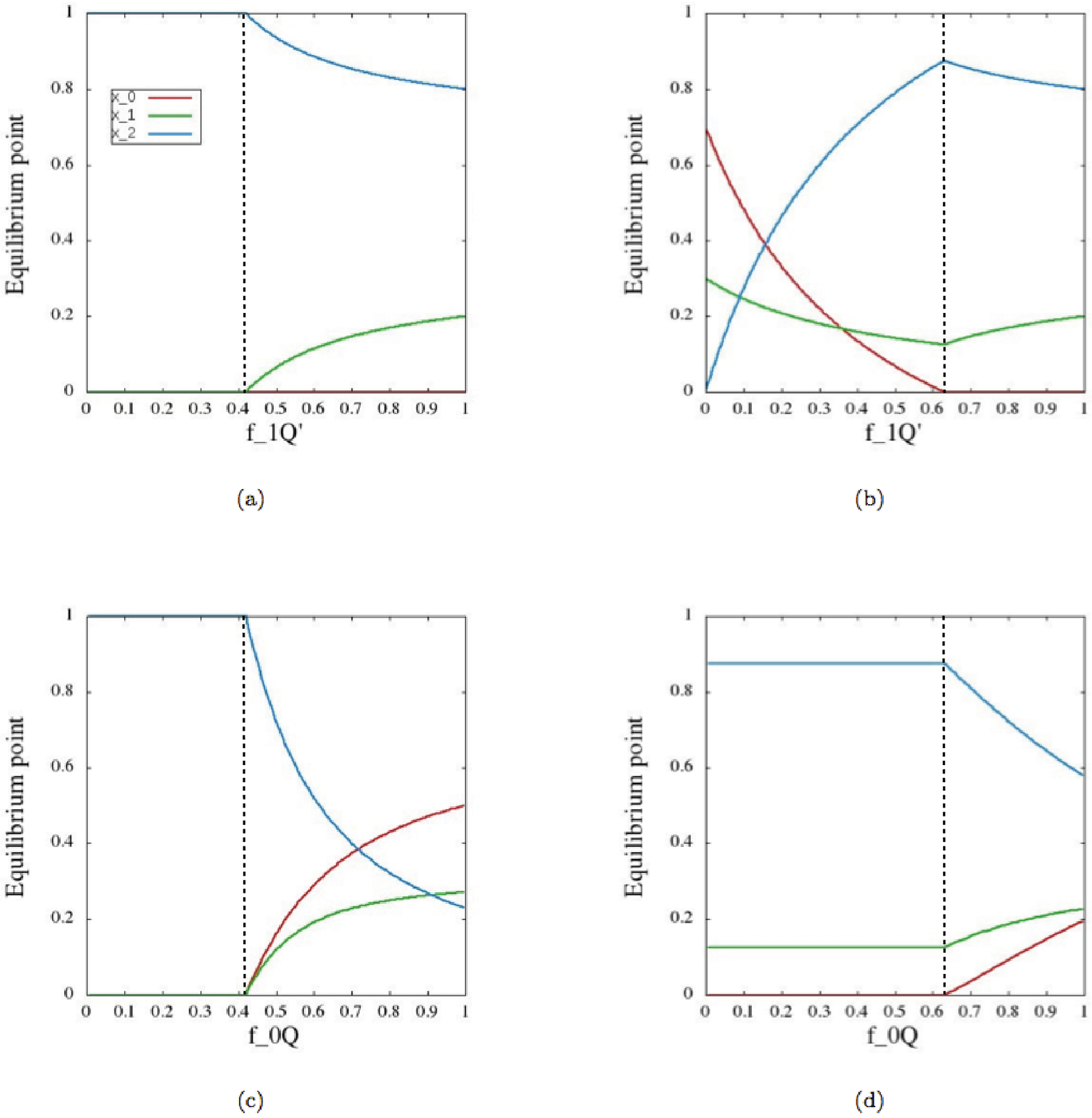}
\vspace{1cm}
\caption{Impact of changing the replication fidelity of the populations $x_0$ and $x_1$, i.e., $f_0Q$ and $f_1Q'$ respectively, in the equilibria of the three cell populations. In (a) and (b) $f_0Q$ has fixed values $0.21$ and $0.63$ respectively; in (c) and (d), $f_1$ has fixed values $0.21$ and $0.63$. Notice that the letter of each panel correspond to sections made in Fig.~\ref{fixpoints}, represented by the white arrows. The vertical dashed lines denote bifurcation values.}
\label{bifurcations}
\end{figure}

As a particular case we consider $Q=0.7$, $Q'=0.3$ and $f_2=0.42$. Then we make $f_0$ to take values from $0.01$ to $\frac{1}{Q}$ and $f_1$ to take values from $0.01$ to $\frac{1}{Q'}$ both with a step of $0.01$. We can not start with $f_0 =0$ or $f_1=0$ because this way $x_0$ or $x_1$ would become extinct. So, we compute the analytical result for all the possible fixed points and, with the iterative method, we integrate the ODE until the distance between the result of the iterate and one of the fixed points is less than an error tolerance previously set, in our case $10^{-16}$. We also fix an upper limit for the time taken by the system to reach a pre-established distance to the fixed point, therefore we consider it does not reach any fixed point if this upper limit is surpassed.


These computations are displayed in Figure~\ref{fixpoints}, which shows the fixed points that are reached by the system depending on the parameters $f_0Q$ and $f_1Q'$. Here each fixed point is indicated with a different colour. Red indicates that the fixed point is $x^* = (0,0,1)$; blue indicates that the fixed point is $x^*=(0,\frac{f_1Q'-f_2}{f_1-f_2},\frac{f_1-f_1Q'}{f_1-f_2})$; and green indicates the fixed point where all of subpopulations coexist. In addition, black lines indicate that none of the fixed points were reached by the system in the given time. Notice that in our analysis we are tuning $f_0Q$ and $f_1Q'$. The decrease of these pair of parameters is qualitatively equivalent to the increase of mutation rates $\mu_0$ and $\mu_1$, respectively, since $\mu_0 = 1-Q$ and $\mu_1 = 1 -Q'$. For instance, going from $f_1Q'=1$ to $f_1Q'=0$ can be achieved increasing $\mu_1$.
\bigskip

Figure~\ref{3d_fix} shows the density of each population at the equilibrium point in the same parameter space of Fig.~\ref{fixpoints}. This analysis allows us to characterize the regions of this parameter space where the density of stable cells $x_0$ is high, and the malignant population $x_2$ remains low. As Sol\'e and Deisboeck \cite{SOLE2001} identified, this scenario can be achieved by increasing $\mu_1$. Furthermore, our results suggest that this behavior is robust to changes in $\mu_0$, whenever $\mu_0$ remains small (i.e., high copying fidelity $f_0 Q$). Notice that an increase of $f_0Q$ (or, alternatively, a decrease of $\mu_0$) makes the population densities of $x_0^*$ and $x_2^*$ to increase and decrease respectively, while for this range of parameters the values of $x_1^*$ remain low. This highlights the relevance of targeted mutagenic therapies against tumor cells, while the most stable cells should keep replicating with a high fidelity. Such a therapy may slow down tumor growth, which could be eventually be cleared due to demographic stochasticity because the small population numbers for $x_2$ found for this combination of parameters (i.e., large $\mu_1$ together with low $\mu_0$).
\begin{figure}[!ht]
\centering
\includegraphics[scale=0.6]{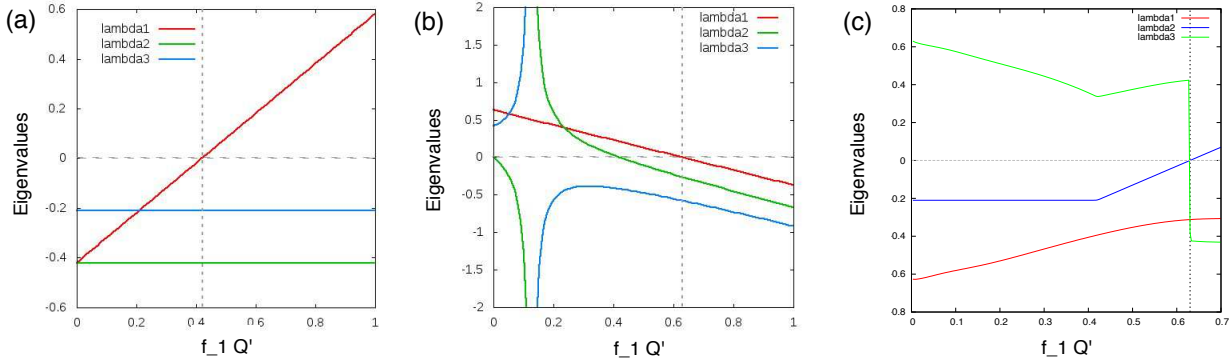}
\caption{(a) Eigenvalues of $L_{\mu}(x^*)$ for $x^*=(0,0,1)$ and $f_0Q = 0.21$, using the range of the arrow (a) in Fig.~\ref{fixpoints}.. (b) Eigenvalues of $L_{\mu}(x^*)$ as a function of $f_1 Q'$ for the fixed point $x^*=(0,x_1^*,x_2^*)$ using $f_0Q=0.63$. (c) Eigenvalues of the differential at the equilibrium point as a function of $f_1Q'$. The vertical dotted line corresponds to the bifurcation value $0.63$.}
\end{figure}

We observe that there is a frontier between the different fixed points reached by the system. The pass from one fixed point to another can be given by a bifurcation. In Fig.~\ref{bifurcations} we appreciate these bifurcations (see dashed vertical lines) more clearly (in the next section we will study the bifurcations in detail). For this purpose, we have set four sections of Fig.~\ref{fixpoints} (white arrows), i.e., we have fixed different values for $f_0$ and $f_1$. In this case these values are $f_0Q=0.21,0.63$ and $f_1Q'=0.21,0.63$.
%

We see the bifurcations of the fixed points of the system in Fig.~\ref{bifurcations}. In $(a)$ and $(b)$ $f_0Q$ has fixed values $0.21$ and $0.63$ respectively, so bifurcations are represented depending on the replication fidelity of $x_1$ in both cases. We notice that for high replication fidelity only $x_1$ and $x_2$ coexist. But in $(b)$ the case is different: if we study this graphic in terms of the mutation rate of $x_1$ we can observe that for high values of it, the system stays at an equilibrium point where $x_0$ is greater. That means that if we make the mutation rate $\mu_1$ higher through therapy it is possible to make the tumor to achieve an equilibrium state where the most unstable cells are near to $0$ and the whole population is mainly dominated by $x_0$.

In cases $(c)$ and $(d)$ we have considered fixed values $f_1Q' = 0.21$ and $f_1Q' = 0.63$ and, therefore, now bifurcations are represented in terms of the replication fidelity of $x_0$. Notice that in both cases the higher is the probability of $x_0$ to stay stable (equivalently, the mutation rate $\mu_0$ of $x_0$ is low) the higher is the population of stable cells at the equilibrium point. This is crucial since they correspond to final scenarios with an important presence of genetically stable cells. Comparing these cases, we observe that, if the mutation rate of $x_1$ is higher, i.e., case (c) ($f_1Q' = 0.21$, i.e., $f_1\mu_1 = 0.49$), the population $x_0$ in the equilibrium point reached is also higher. This suggests the existence of a threshold in the unstability of $x_1$ bringing more stable cells $x_0$ into the final equilibria.

\subsection{Stability analysis and bifurcations}
The linear stability analysis of the fixed points characterized in the previous Section is performed by using the Jacobi matrix:

\begin{eqnarray}
L_\mu(x^*)=
\begin{pmatrix}
f_0Q - \Phi(\textbf{x})-x_0f_0 & -x_0f_1 & -x_0f_2 \\
f_0(1-Q) -x_1f_0 & f_1Q' -\Phi(\textbf{x}) -x_1f_1& -x_1f_2\\
-x_2f_0 & f_1(1-Q') -x_2f_1& f_2 - \Phi(\textbf{x}) -x_2f_2
\end{pmatrix}.
\label{jac}
\end{eqnarray}\\

We are specially concerned with the domain, in the parameter space, where the malignant cells become dominant and the stability of such equilibrium state. Thus, taking $x^* = (0,0,1)$, the Jacobi matrix has the following eigenvalues:
\begin{eqnarray}
\lambda_1 = f_0Q-f_2, \phantom{xx} \lambda_2 = f_1Q'-f_2, \phantom{x} {\rm{and}} \phantom{x} \lambda_3 = -f_2.
\end{eqnarray}
So, $x^*$ is an attractor if the two inequalities, $f_0Q < f_2$ and $f_1Q' < f_2$, are satisfied. From the expression of the eigenvalues we conclude that there is a critical condition for the mutation rates $\mu_0 = 1-Q$ and $\mu_1 = 1-Q'$, given by: $\mu_0^c = 1-\frac{f_2}{f_0}$ and $\mu_1^c = \frac{f_2}{f_1}$, respectively (see also \cite{SOLE2004}). These conditions separate the domain where only $x_2$ remains from the other two cases. From Fig. 5(a) we can confirm that the fixed point $x^* = (0,0,1)$ is an attractor point until this critical condition i.e., the error threshold, (shown by a vertical dotted line) is reached. From that value, the fixed point $x^*$ is unstable, with local $2$-dimensional stable invariant manifold and a $1$-dimensional unstable one.
\begin{figure}[!ht]
\centering
\includegraphics[width=15cm]{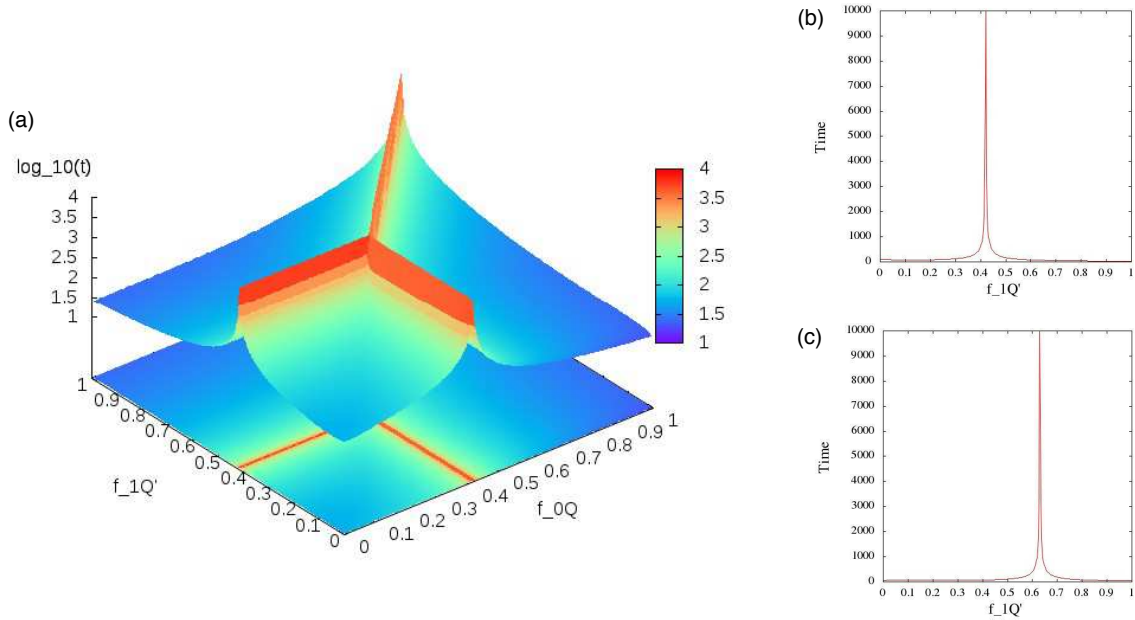}
\caption{Time taken by the system to reach a distance $10^{-2}$ to a fixed point in the parameter space ($f_0Q$,$f_1Q'$). Time in (a) is logarithmic represented in base 10 but in sections (b) and (c) it is represented in real scale due to appreciate the actual velocity of the system.}
\label{time}
\end{figure}
\begin{eqnarray}
\lambda_1 = f_0Q-f_1Q', \phantom{xx} \lambda_2 = \frac{f_1(f_2-f_1Q')}{f_1-f_2}, \phantom{x} {\rm{and}} \phantom{x} \lambda_3 = \frac{2f_1f_2Q'-f_1^2Q'-f_2^2}{f_1-f_2}. 
\end{eqnarray}

If we evaluate the stability of the fixed point $x^*=(0,\frac{f_1Q'-f_2}{f_1-f_2},\frac{f_1-f_1Q'}{f_1-f_2})$, where populations $x_1$ and $x_2$ coexist, we get the following eigenvalues from the Jacobi matrix:


Notice that, as we mention in Remark~\ref{f1f2}, when $f_1 = f_2$, such fixed point does not exist, hence the eigenvalues of the Jacobi matrix do not provide any information. This can be appreciated in Fig. 5(b), when $f_1Q' = 0.126$, thus $f_1 = f_2 = 0.42$. 

Finally, we want to study the stability of the fixed point where all populations coexist. An analytical expression for the eigenvalues exists, since we have the Jacobi matrix and the analytical expression for the fixed point itself. But they have really complicated expressions and, even more, we are considering $n=1$. This means that the greater $n$, the analytical expressions for the eigenvalues are more complicated to find. This is the reason why it is interesting to compute them numerically.

To compute the eigenvalues of this $3$-dimensional matrix we proceed in the following way. This simplified procedure turns to be quite fast in our case (dimension $3$) but may not be applicable for higher dimensions. It works as follows: 
\begin{itemize}
\item We first apply the \emph{power method} to compute an approximation for the eigenvalue of maximal modulus. As it is known, it is based on the idea that the sequence 
$v_{k+1}=A v_k$ should behave for large values of $k$ as the direction of the eigenvector associated to the largest (in modulus) eigenvalue $\lambda_{\max}$ of $A$. To compute it one starts from an arbitrary initial vector $v_0$  (in our case, for instance, $(1/3, 1/3, 1/3)$, and computes $v_{k+1}=A v_k$. Provided the difference between the largest eigenvalue and the second largest eigenvalue (in modulus, always) is not too small, the quotients of the components of $v_{k+1}, v_{k}$, that is $v_{k+1}^{(j+1)}/v_{k+1}^{(j)}$, converge to $\lambda_{\max}$. To avoid problems of overflow one normalizes $v_{k+1}$, i.e. $w_{k+1}=v_{k+1}/\parallel v_{k+1}\parallel$, at any step.
Problems of convergence appear when the two largest eigenvalues of $A$ are close.

\item We apply the same method to $A^{-1}$ to obtain the smallest eigenvalue of $A$, since $\lambda$ is eigenvalue of $A$ iff $\lambda^{-1}$ is eigenvalue of $A^{-1}$.
To compute $A^{-1}$ we have used its $QR$-decomposition, which writes $A$ as the product of two matrices: an orthogonal matrix $Q$ and an upper triangular matrix $R$, i.e., 
$A = QR$.

\item Provided $\lambda_{\max}$, $\lambda_{\min}$ are accurate enough, the third eigenvalue is determined from the value of the determinant of the matrix $A$, that we have derived from its $QR$-decomposition.  
   
\end{itemize}

\begin{figure}[!ht]
\centering
\includegraphics[scale=0.48]{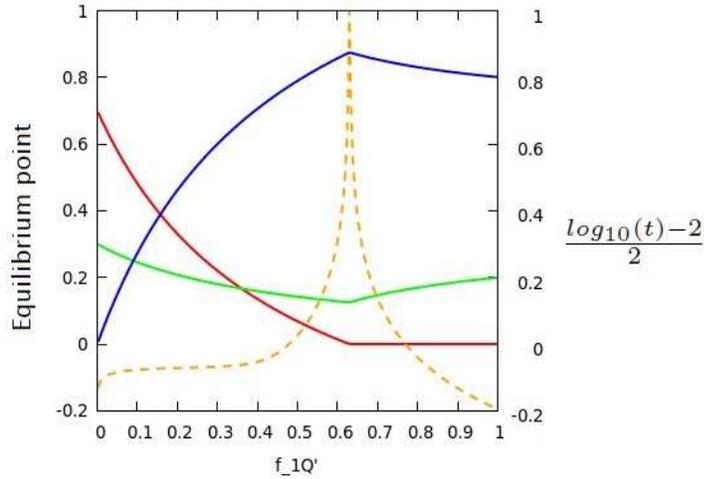}
\caption{Transient times dependence on $f_1 Q'$. The dashed line corresponds to $\frac{\log_{10}(t)-2}{2}$, and the solid lines are the equilibrium densities of $x_0$ (red), $x_1$ (green) and $x_2$ (blue) for $f_0Q = 0.63$. Notice that time is maximum at the bifurcation value.}
\label{merge}
\end{figure}

We apply this procedure to the study of possible bifurcations in the stability of the equilibrium points obtained when we do not have analytic expressions for the associated eigenvalues. To show it, we fix a value for $f_0 Q$ and move $f_1 Q'$. We have chosen a value for $f_0 Q$ corresponding to the line $(b)$ of the diagram in Fig~\ref{fixpoints}. Observe that, when moving at the green zone our equilibrium point is of the form $(x_0^*, x_1^*, 1-x_0^*-x_1^*)$ while it is of the form $(0,x_1^*,1-x_1^*)$ when we move along the blue one. For each equilibrium point, whose components depend on $f_1 Q'$, we compute its differential matrix and its associated eigenvalues. These eigenvalues, computed numerically as mentioned above, are plotted in Fig. 5(c). They are all three real, and thus no transient oscillations are found in the dynamics. Observe that there is an interchange of the number of positive and negative eigenvalues around the bifurcation value $0.63$, but no change in their stability. They are always unstable since we have at least one positive eigenvalue. 

It is interesting to highlight the change in the geometry of such equilibrium point: for values of $f_1 Q'$ under the bifurcation value $0.63$ its invariant stable manifold is $2$-dimensional and its invariant unstable manifold $1$-dimensional; on the contrary, after the bifurcation the associated  dimensions of the invariant manifolds are exchanged. In both cases, the orbit starting at initial conditions $(1,0,0)$ finishes at the stable manifold of that equilibrium point. This is why, although being a unstable fixed point, the trajectory under analysis asymptotically travels towards the fixed point.

\subsection{Transient times}
In this section we analyze the time taken by the system to reach a given (small) distance to one of the fixed points found in section \ref{se:eqpoints}. Typically, the behaviour of transient times change near bifurcation threshold, and, particularly for our system, we are interested in possible changes in transients due to changes in mutation rates. These phenomena could be relevant in patient response under mutagenic therapy.
%
%

Figure 7 shows the same as Fig.~\ref{fixpoints}, in the sense that we tune the same parameters of the system, but now we compute the  transient times. Here we use the same method to integrate the ODE system, but taking into account only the time taken to reach a distance $10^{-2}$ of the equilibrium point. Notice that if we observe the contour lines represented at the bottom of Fig.~\ref{time} we have the same as in Fig.~\ref{fixpoints}. This means that the ``time to equilibrium" of the system increases when we are near bifurcation points, which are represented with black lines in Fig.~\ref{fixpoints}.

When slightly modifying the mutation rate of $x_1$ near the error threshold, we see that the time taken by the tumor to reach an equilibrium state is significantly lower. In particular, when increasing the mutation rate, not only time decreases but also we see that the tumor reaches an equilibrium point where the stable cells population is higher. This can be appreciated in Fig.~\ref{merge}, where red, green and blue lines represent $x_0$, $x_1$ and $x_2$ respectively and the dashed, orange line represents the transient time, which is rescaled to be able to relate it with the corresponding equilibrium point. In case of possible mutagenic therapies directly targeting the most unstable cells, our results reveal that pushing mutation rates of unstable cells beyond the error threshold (where $x_0$ population starts increasing and $x_2$ population starts decreasing) the achievement of the equilibrium point would be very fast, and no further increase of mutation rates may be needed to achieve a faster response (notice that the transient time for $f_1 Q' \lessapprox 0.45$ in Fig.~\ref{time} does not significantly decrease for smaller $f_1 Q'$ values). In other words, when rescaling Fig.~\ref{time} (b) in terms of logarithm in base 10 (as we can see in Fig.~\ref{merge}) if we increase the mutation rate of $x_1$, time decreases faster than an exponential function, which is interesting result since it implies an equilibrium scenario with a dominant presence of $x_0$ that can be reached in a short time.

\section{Conclusions}
In this article we study an ODE system modeling the behavior of a population of tumor cells with a mean-field quasispecies model introduced by Sol\'e and Deisboeck \cite{SOLE2001}. Thus, in our simplified model we have assumed the following: the model does not consider stochasticity and does not take into account spatial correlations between cells as well. Also, we do not consider cell death (due to e.g., the immune system), otherwise we model competition in terms of replication and mutation using the quasispecies framework \cite{Eigen1971,Eigen1979}.

First, we have found the fixed points of the system analytically and we have studied their stability both analytically and numerically. We have also characterized the bifurcations between the different fixed points. We conclude that, depending on the parameters, the system can reach different equilibrium states one of them involving low populations of tumor cells while keeping large populations of genetically stable cells. This scenario can be achieved by increasing the level of genetic instability conferred through the mutation rate $\mu_1$ of the mutator-phenotype population $x_1$.
If this mutation rate exceeds an error threshold, the replication rate of the more malignant subpopulation $x_2$ is reduced to a point where it has no longer competitive advantage.

Further analysis of the effects of mutation rate $\mu_1$ on the dynamics of the system have allowed us to characterize a transcritical bifurcation. Under such a bifurcation, the time taken by the system to reach the desired equilibrium state is shown to drastically decrease with slight changes in the mutation rate $\mu_1$ near the error threshold. This result can give us a clue about how medication and therapy may affect the tumor behavior in case of direct mutagenic therapies against tumor cells. In other words, it is possible to modify the mutation rate of the cells through therapy, and we have seen that if we slightly increase the mutation rate of $x_1$ near the error threshold we can obtain an equilibrium point with a dominant population of stable cells rapidly. Such an scenario could become relevant since populations of tumor cells could be decreased with mutagenic therapy, and such populations could be more prone to extinction due to stochastic fluctuations inherent to small populations sizes. 

Further research should explicitly consider stochastic fluctuations by transforming the differential equations in stochastic differential equations e.g., Langevin equations. Also, the addition of spatial correlations would indicate if our results are still present in theoretical models for solid tumors. In this sense, it would be interesting to analyze the transient times tied to increasing mutation values considering space in an explicit manner. It is important to notice that transient times are extremely important in terms of responses velocities after therapy. Different causes and mechanisms have been suggested to be involved in transient duration, which have been especially explored for ecological systems \cite{transient}. For instance, it is known that space can involve very long transients until equilibrium points are reached. This phenomenon could be investigated extending the model we investigated by means of partial differential equations.

\section*{Acknowledgements}
J.T.L. has been partially supported by the Spanish MICIIN/FEDER grant MTM2012-31714 and by the Generalitat de Catalunya number 2014SGR-504. J.S. has been funded by the Fundaci\'on Bot\'in.



\end{document}